\documentclass[12pt,oneside]{amsart}  
\pdfoutput=1

\usepackage{amssymb}
\usepackage{amsfonts}
\usepackage{amsthm}
\usepackage{amsthm}
\usepackage{amscd}
\usepackage{wrapfig}
\usepackage[matrix, arrow, curve]{xy}
\usepackage{graphicx}

\reversemarginpar         
\numberwithin{equation}{section}
\theoremstyle{plain}
\newtheorem{theorem}{Theorem}[section]

\theoremstyle{definition}
\newtheorem{definition}[theorem]{Definition}
\theoremstyle{remark}
\newtheorem*{remark}{Remark}

\begin{document}
%

\newcommand{\MgNekp}{\mathcal{M}_{g,N+1}^{(k,p)}} 
\newcommand{\M}{\mathcal{M}_{g,N+1}^{(1)}}
\newcommand{\Teich}{\mathcal{T}_{g,N+1}^{(1)}}
\newcommand{\T}{\mathrm{T}}
\newcommand{\corr}{\bf}
\newcommand{\vac}{|0\rangle}
\newcommand{\Ga}{\Gamma}
\newcommand{\new}{\bf}
\newcommand{\define}{\def}
\newcommand{\redefine}{\def}
\newcommand{\Cal}[1]{\mathcal{#1}}
\renewcommand{\frak}[1]{\mathfrak{{#1}}}
\newcommand{\refE}[1]{(\ref{E:#1})}
\newcommand{\refS}[1]{Section~\ref{S:#1}}
\newcommand{\refSS}[1]{Section~\ref{SS:#1}}
\newcommand{\refT}[1]{Theorem~\ref{T:#1}}
\newcommand{\refO}[1]{Observation~\ref{O:#1}}
\newcommand{\refP}[1]{Proposition~\ref{P:#1}}
\newcommand{\refD}[1]{Definition~\ref{D:#1}}
\newcommand{\refC}[1]{Corollary~\ref{C:#1}}
\newcommand{\refL}[1]{Lemma~\ref{L:#1}}
\newcommand{\R}{\ensuremath{\mathbb{R}}}
\newcommand{\C}{\ensuremath{\mathbb{C}}}
\newcommand{\N}{\ensuremath{\mathbb{N}}}
\newcommand{\Q}{\ensuremath{\mathbb{Q}}}
\renewcommand{\P}{\ensuremath{\mathcal{P}}}
\newcommand{\Z}{\ensuremath{\mathbb{Z}}}
\newcommand{\kv}{{k^{\vee}}}
\renewcommand{\l}{\lambda}
\newcommand{\gb}{\overline{\mathfrak{g}}}
\newcommand{\hb}{\overline{\mathfrak{h}}}
\newcommand{\g}{\mathfrak{g}}
\newcommand{\h}{\mathfrak{h}}
\newcommand{\gh}{\widehat{\mathfrak{g}}}
\newcommand{\ghN}{\widehat{\mathfrak{g}_{(N)}}}
\newcommand{\gbN}{\overline{\mathfrak{g}_{(N)}}}
\newcommand{\tr}{\mathrm{tr}}
\newcommand{\sln}{\mathfrak{sl}}
\newcommand{\sn}{\mathfrak{s}}
\newcommand{\so}{\mathfrak{so}}
\newcommand{\spn}{\mathfrak{sp}}
\newcommand{\tsp}{\mathfrak{tsp}(2n)}
\newcommand{\gl}{\mathfrak{gl}}
\newcommand{\slnb}{{\overline{\mathfrak{sl}}}}
\newcommand{\snb}{{\overline{\mathfrak{s}}}}
\newcommand{\sob}{{\overline{\mathfrak{so}}}}
\newcommand{\spnb}{{\overline{\mathfrak{sp}}}}
\newcommand{\glb}{{\overline{\mathfrak{gl}}}}
\newcommand{\Hwft}{\mathcal{H}_{F,\tau}}
\newcommand{\Hwftm}{\mathcal{H}_{F,\tau}^{(m)}}

\newcommand{\car}{{\mathfrak{h}}}    
\newcommand{\bor}{{\mathfrak{b}}}    
\newcommand{\nil}{{\mathfrak{n}}}    
\newcommand{\vp}{{\varphi}}
\newcommand{\bh}{\widehat{\mathfrak{b}}}  
\newcommand{\bb}{\overline{\mathfrak{b}}}  
\newcommand{\Vh}{\widehat{\mathcal V}}
\newcommand{\KZ}{Kniz\-hnik-Zamo\-lod\-chi\-kov}
\newcommand{\TUY}{Tsuchia, Ueno  and Yamada}
\newcommand{\KN} {Kri\-che\-ver-Novi\-kov}
\newcommand{\pN}{\ensuremath{(P_1,P_2,\ldots,P_N)}}
\newcommand{\xN}{\ensuremath{(\xi_1,\xi_2,\ldots,\xi_N)}}
\newcommand{\lN}{\ensuremath{(\lambda_1,\lambda_2,\ldots,\lambda_N)}}
\newcommand{\iN}{\ensuremath{1,\ldots, N}}
\newcommand{\iNf}{\ensuremath{1,\ldots, N,\infty}}

\newcommand{\tb}{\tilde \beta}
\newcommand{\tk}{\tilde \kappa}
\newcommand{\ka}{\kappa}
\renewcommand{\k}{\kappa}

\newcommand{\Pif} {P_{\infty}}
\newcommand{\Pinf} {P_{\infty}}
\newcommand{\PN}{\ensuremath{\{P_1,P_2,\ldots,P_N\}}}
\newcommand{\PNi}{\ensuremath{\{P_1,P_2,\ldots,P_N,P_\infty\}}}
\newcommand{\Fln}[1][n]{F_{#1}^\lambda}
\newcommand{\tang}{\mathrm{T}}
\newcommand{\Kl}[1][\lambda]{\can^{#1}}
\newcommand{\A}{\mathcal{A}}
\newcommand{\U}{\mathcal{U}}
\newcommand{\V}{\mathcal{V}}
\renewcommand{\O}{\mathcal{O}}
\newcommand{\Ae}{\widehat{\mathcal{A}}}
\newcommand{\Ah}{\widehat{\mathcal{A}}}
\newcommand{\La}{\mathcal{L}}
\newcommand{\Le}{\widehat{\mathcal{L}}}
\newcommand{\Lh}{\widehat{\mathcal{L}}}
\newcommand{\eh}{\widehat{e}}
\newcommand{\Da}{\mathcal{D}}
\newcommand{\kndual}[2]{\langle #1,#2\rangle}
\newcommand{\cins}{\frac 1{2\pi\mathrm{i}}\int_{C_S}}
\newcommand{\cinsl}{\frac 1{24\pi\mathrm{i}}\int_{C_S}}
\newcommand{\cinc}[1]{\frac 1{2\pi\mathrm{i}}\int_{#1}}
\newcommand{\cintl}[1]{\frac 1{24\pi\mathrm{i}}\int_{#1 }}
\newcommand{\w}{\omega}
\newcommand{\ord}{\operatorname{ord}}
\newcommand{\res}{\operatorname{res}}
\newcommand{\nord}[1]{:\mkern-5mu{#1}\mkern-5mu:}
\newcommand{\ad}{\operatorname{ad}}
\newcommand{\codim}{\operatorname{codim}}
\newcommand{\Fn}[1][\lambda]{\mathcal{F}^{#1}}
\newcommand{\Fl}[1][\lambda]{\mathcal{F}^{#1}}
\renewcommand{\Re}{\mathrm{Re}}

\newcommand{\ha}{H^\alpha}

\define\ldot{\hskip 1pt.\hskip 1pt}
\define\ifft{\qquad\text{if and only if}\qquad}
\define\a{\alpha}
\redefine\d{\delta}
\define\w{\omega}
\define\ep{\epsilon}
\redefine\b{\beta} \redefine\t{\tau} \redefine\i{{\,\mathrm{i}}\,}
\define\ga{\gamma}
\define\cint #1{\frac 1{2\pi\i}\int_{C_{#1}}}
\define\cintta{\frac 1{2\pi\i}\int_{C_{\tau}}}
\define\cintt{\frac 1{2\pi\i}\oint_{C}}
\define\cinttp{\frac 1{2\pi\i}\int_{C_{\tau'}}}
\define\cinto{\frac 1{2\pi\i}\int_{C_{0}}}
\define\cinttt{\frac 1{24\pi\i}\int_C}
\define\cintd{\frac 1{(2\pi \i)^2}\iint\limits_{C_{\tau}\,C_{\tau'}}}
\define\cintdr{\frac 1{(2\pi \i)^3}\int_{C_{\tau}}\int_{C_{\tau'}}
\int_{C_{\tau''}}}
\define\im{\operatorname{Im}}
\define\re{\operatorname{Re}}
\define\res{\operatorname{res}}
\redefine\deg{\operatornamewithlimits{deg}}
\define\ord{\operatorname{ord}}
\define\rank{\operatorname{rank}}
\define\fpz{\frac {d }{dz}}
\define\dzl{\,{dz}^\l}
\define\pfz#1{\frac {d#1}{dz}}

\define\K{\Cal K}
\define\U{\Cal U}
\redefine\O{\Cal O}
\define\He{\text{\rm H}^1}
\redefine\H{{\mathrm{H}}}
\define\Ho{\text{\rm H}^0}
\define\A{\Cal A}
\define\Do{\Cal D^{1}}
\define\Dh{\widehat{\mathcal{D}}^{1}}
\redefine\L{\Cal L}
\newcommand{\ND}{\ensuremath{\mathcal{N}^D}}
\redefine\D{\Cal D^{1}}
\define\KN {Kri\-che\-ver-Novi\-kov}
\define\Pif {{P_{\infty}}}
\define\Uif {{U_{\infty}}}
\define\Uifs {{U_{\infty}^*}}
\define\KM {Kac-Moody}
\define\Fln{\Cal F^\lambda_n}
\define\gb{\overline{\mathfrak{ g}}}
\define\G{\overline{\mathfrak{ g}}}
\define\Gb{\overline{\mathfrak{ g}}}
\redefine\g{\mathfrak{ g}}
\define\Gh{\widehat{\mathfrak{ g}}}
\define\gh{\widehat{\mathfrak{ g}}}
\define\Ah{\widehat{\Cal A}}
\define\Lh{\widehat{\Cal L}}
\define\Ugh{\Cal U(\Gh)}
\define\Xh{\hat X}
\define\Tld{...}
\define\iN{i=1,\ldots,N}
\define\iNi{i=1,\ldots,N,\infty}
\define\pN{p=1,\ldots,N}
\define\pNi{p=1,\ldots,N,\infty}
\define\de{\delta}

\define\kndual#1#2{\langle #1,#2\rangle}
\define \nord #1{:\mkern-5mu{#1}\mkern-5mu:}
\define \sinf{{\widehat{\sigma}}_\infty}
\define\Wt{\widetilde{W}}
\define\St{\widetilde{S}}
\newcommand{\SigmaT}{\widetilde{\Sigma}}
\newcommand{\hT}{\widetilde{\frak h}}
\define\Wn{W^{(1)}}
\define\Wtn{\widetilde{W}^{(1)}}
\define\btn{\tilde b^{(1)}}
\define\bt{\tilde b}
\define\bn{b^{(1)}}
%
\define\eps{\varepsilon}    
\define\doint{({\frac 1{2\pi\i}})^2\oint\limits _{C_0}
       \oint\limits _{C_0}}                            
\define\noint{ {\frac 1{2\pi\i}} \oint}   
\define \fh{{\frak h}}     
\define \fg{{\frak g}}     
\define \GKN{{\Cal G}}   
\define \gaff{{\hat\frak g}}   
\define\V{\Cal V}
\define \ms{{\Cal M}_{g,N}} 
\define \mse{{\Cal M}_{g,N+1}} 
\define \tOmega{\Tilde\Omega}
\define \tw{\Tilde\omega}
\define \hw{\hat\omega}
\define \s{\sigma}
\define \car{{\frak h}}    
\define \bor{{\frak b}}    
\define \nil{{\frak n}}    
\define \vp{{\varphi}}
\define\bh{\widehat{\frak b}}  
\define\bb{\overline{\frak b}}  
\define\Vh{\widehat V}
\define\KZ{Knizhnik-Zamolodchikov}
\define\ai{{\alpha(i)}}
\define\ak{{\alpha(k)}}
\define\aj{{\alpha(j)}}
\newcommand{\laxgl}{\overline{\mathfrak{gl}}}
\newcommand{\laxsl}{\overline{\mathfrak{sl}}}
\newcommand{\laxso}{\overline{\mathfrak{so}}}
\newcommand{\laxsp}{\overline{\mathfrak{sp}}}
\newcommand{\laxs}{\overline{\mathfrak{s}}}
\newcommand{\laxg}{\overline{\frak g}}
\newcommand{\bgl}{\laxgl(n)}
\newcommand{\tX}{\widetilde{X}}
\newcommand{\tY}{\widetilde{Y}}
\newcommand{\tZ}{\widetilde{Z}}

\vspace*{-1cm}
%
%
%
\vspace*{2cm}

\large{
\title[Lax operator algebras and gradings]
{Lax operator algebras and gradings on semi-simple Lie algebras}
\author[O.K. Sheinman]{Oleg K. Sheinman}
\thanks{The work of O.K.Sheinman has been supported in part
by the program "Fundamental Problems of Nonlinear Dynamics" of the
Russian Academy of Sciences, and by the RFBR projects 14-01-00012-a, 13-01-12469 ofi-m2}

\begin{abstract}
A Lax operator algebra is constructed for an arbitrary semi-simple Lie algebra over $\C$ equipped with a $\Z$-grading, and arbitrary compact Riemann surface with marked points. In this set-up, a treatment of almost graded structures, and classification of the central extensions of Lax operator algebras are given. A relation to the earlier approach based on the Tyurin parameters is established.
\end{abstract} 
\maketitle

\tableofcontents

\section{Introduction}\label{S:intro}

Lax operator algebras have been introduced in \cite{KSlax} in connection with the notion of Lax operator with a spectral parameter on a Riemann surface due to I.Krichever \cite{Klax}. They constitute a certain class of current algebras on Riemann surfaces closely related to finite-dimensional integrable systems like Hitchin systems (including their version for
pointed Riemann surfaces), integrable gyroscopes and integrable cases of two-dimensional hydrodynamics of a solid body. The formalism of Lax operator algebras provides a general and transparent treatment (inherent in principle in \cite{Klax}) of Hamiltonian theory of the corresponding Lax equations.

In certain respects Lax operator algebras are similar to Kac--Moody algebras. They possess an almost graded structure, and a theory of central extensions very similar to that of their prototype, though technically more complicated. The state of the theory of Lax operator algebras and their applications actual to the end of 2013 has been summarized in \cite{Sh_DGr}.

So far Lax operator algebras have been constructed only for classical simple (and some reductive) Lie algebras over $\C$ (serving as the range of values of currents), and for the exceptional Lie algebra $G_2$ \cite{KSlax,Sh_DGr,Sh_G2}, in terms of their matrix models. The proofs of their properties were specific for every type of reductive Lie algebra, and quite technically involved, especially for $\spn(2n)$ and $G_2$. The question whether there exists any general approach based on the theory of root systems was open though had been posed many times by the author (see \cite{Sh_G2} for example).

An important ingredient, the definition of Lax operator algebras is based on, is given by Tyurin parameters. These are the data classifying holomorphic vector bundles on compact complex Riemann surfaces by the theorem due to A.N.Tyurin \cite{Tyvb}. These data consist of a set of marked points on the Riemann surface (called \emph{Tyurin points}, or \emph{$\ga$-points} below), and a set of associated with them elements of a projective space. A local part of the definition of Lax operator algebras postulates a certain form, and properties of Laurent expansions of their elements at the Tyurin points.

Recently E.B.Vinberg has drawn the author's attention to the fact that the local conditions at the Tyurin points can be formulated in terms of an arbitrary semi-simple Lie algebra, and its $\Z$-grading. Based on this observation, we define here Lax operator algebras for an arbitrary semi-simple Lie algebra. A detailed description of graded structures on semi-simple Lie algebras in terms of their root systems, and their matrix or tensor models as well, is given in \cite{Vin}. To a great extent, this description is based on the works of E.B.V. and his collaborators.

The new approach unifies and crucially simplifies all proofs, and enables us to construct new Lax operator algebras. For the classic Lie algebras, and for $G_2$, the Tyurin parameters, or their analogs, arise in the new approach automatically.

The two main results of the paper are represented by Theorems \ref{T:almgrad}, and \ref{T:central} formulated and proved in the \refS{constr}, and \refS{cexst}, respectively. The first of them gives a description of the Lie algebra structure, and almost graded structures on Lax operator algebras, while the second gives a construction and classification of their central extensions. In the \refS{Tyupar}, we consider a number of examples, including all classical root systems, and $G_2$, mainly in order to state a correspondence with the earlier approach, in particular, show an appearance of Tyurin parameters.

The author is grateful to E.B.Vinberg for illuminating discussions.

\section{The current algebra and its almost graded structures}\label{S:constr}
Let $\g$ be a semi-simple Lie algebra over $\C$, $\h$ its Cartan subalgebra, $h\in\h$ such that $p_i=\a_i(h)\in\Z_+$ for any simple root $\a_i$ of $\g$. Let $\g_p=\{ X\in\g\ |\ (\ad h)X=pX \}$, and $k=\max\{p\ |\ \g_p\ne 0\}$. Then $\g=\bigoplus\limits_{i=-k}^{k}\g_p$ gives a $\Z$-grading on $\g$. For the theory and classification results on such kind of gradings we refer to \cite{Vin}. We call $k$ the depth of the grading. Obviously, $\g_p=\bigoplus\limits_{\substack{\a\in R\\ \a(h)=p}}\g_\a$ where $R$ is the root system of $\g$. Define also the following filtration of $\g$:  $\tilde\g_p=\bigoplus\limits_{q=-k}^m\g_q$. Then, $\tilde\g_p\subset\tilde\g_{p+1}$ ($p\ge -k$), $\tilde\g_{-k}=\g_{-k},\ldots,\tilde\g_k=\g$, $\tilde\g_p=\g$, $p>k$.

Let $\Sigma$ be a complex compact Riemann surface with two fixed finite sets of marked points: $\Pi$, and
$\Gamma$. Let $L$ be a meromorphic map $\Sigma\to\g$ which is holomorphic outside the marked points, may have poles of arbitrary orders at the points in $\Pi$, and has expansions of the following form at the points in $\Gamma$:
\begin{equation}\label{E:ga_expan}
   L(z)=\bigoplus\limits_{p\ge -k} L_pz^p,\ L_p\in\tilde\g_p
\end{equation}
where $z$ is a local coordinate in a neighborhood of $\ga\in\Gamma$. The grading element $h$ may vary from one $\ga$ to another. For simplicity we assume that $k$ is constant at all $\ga\in\Gamma$ though nothing would change below if we did not assume that.

Let us denote the linear space of all such maps by $\L$. Since the relation \refE{ga_expan} holds true under commutator, $\L$ is a Lie algebra. We fix this important fact as assertion $1^\circ$ of \refT{almgrad} below. The Lie agebra $\L$, its almost-graded structure, and its central extensions are the main subjects of the present paper.  Sometimes we will use the notation $\gb$ instead $\L$ in order to stress its relation to a given semi-simple algebra $\g$. We also keep the name {\it Lax operator algebras} for this class of current algebras in order to emphasize their succession to those in \cite{KSlax,Sh_DGr}.
\begin{definition}
 Given a Lie algebra $\L$, by an {\it almost graded structure} on it we mean a system of its finite-dimensional subspaces $\L_m$, and two non-negative integers $R$, $S$ such that $\L =\bigoplus\limits_{m=-\infty}^\infty \L_m$, and $[\L_m,\L_n]\subseteq\bigoplus\limits_{r=m+n-R}^{m+n+S}\L_r$ ($R$, $S$ are independent of $m$, $n$).
\end{definition}
The almost graded structure on associative and Lie algebras has been introduced by I.M.Krichever and S.P.Novikov  in \cite{KNFa}. For the Lax operator algebras and  the two point case it has been investigated in \cite{KSlax}. Most general setup for both Krichever--Novikov and Lax operator algebras has been considered by M.Schlichenmaier \cite{SLa,SLb,Sch_Laxmulti}.

The above introduced $\L$ possesses a number of almost graded structures. To define one, let us give a splitting of $\Pi$ to a join of two subsets: $\Pi=\{ P_i\ |\ i=1,\ldots,N \}\cup \{ Q_j\ |\ j=1,\ldots,M \}$. Following the lines of \cite{SLa,SLb,Sch_Laxmulti} for every $m\in\Z$ consider three divisors:
\begin{equation}\label{E:3div}
 D^P_m=-m\sum\limits_{i=1}^NP_i,\ D^Q_m=\sum_{j=1}^M(a_jm+b_{m,j})\,Q_j,\ D^\Gamma=k\sum_{\ga\in\Gamma}\ga
\end{equation}
where $a_j,b_{m,j}\in\Q$, $a_j>0$, $a_jm+b_{m,j}$ is an ascending $\Z$\ -valued function of $m$, and there exists a $B\in\R_+$ such that
$|b_{m,j}|\le B, \forall m\in\Z, j=1,\ldots,M$\label{bjmbound}.
We require that
\begin{equation}\label{E:condd}
\sum_{j=1}^Ma_j=N,\qquad
\sum_{i=j}^Mb_{m,j}=N+g-1.
\end{equation}
Let
\begin{equation}\label{E:Dm}
  D_m=D_m^P+D_m^Q+D^\Gamma,
\end{equation}
and
\begin{equation}\label{E:homos}
   \L_m=\{L\in\L\ |(L)+D_m\ge 0\},
\end{equation}
where $(L)$ is the divisor of a $\g$-valued function $L$.
To be more specific of $(L)$, let us notice that by order of a
meromorphic vector-valued function we mean the minimal order of
its entries.

We call $\L_m$ {\it the (homogeneous, grading) subspace of degree
$m$}\index{Homogeneous element (subspace)} of the Lie algebra
$\L$.
\begin{theorem}\label{T:almgrad}
\begin{itemize}
  \item[]
  \item[$1^\circ $] $\L$ is closed with respect to the point-wise commutator $[L,L'](P)=[L(P),L'(P)]$ $(P\in\Sigma)$.
  \item[$2^\circ$] $\dim\,\L_m=N\dim\,\g$;
  \item[$3^\circ$] $\L =\bigoplus\limits_{m=-\infty}^\infty \L_m$ ;
  \item[$4^\circ$] $[\L_m,\L_n]\subseteq\bigoplus\limits_{r=m+n}^{m+n+ g}\L_r$.
\end{itemize}
\end{theorem}
\begin{proof}
A proof of the assertion  $1^\circ $ is obvious as it was already noticed. For the proof of the assertions $3^\circ$, $4^\circ$ we refer to \cite{Sch_Laxmulti} (where they are given for classical Lie algebras but acually hold true in our present set-up). The proof of the assertion $2^\circ$ is the only specific in our set-up, and will be given here.

Let $L(D_m)=\{L\ |\ (L)+D_m\ge 0\}$ where $L:\Sigma\to\g$ is meromorphic but the requirement $L\in\L$ has been relaxed. Let $l_m=\dim L(D_m)$. For the points in $\Pi$, $\Gamma$ in a generic position $l_m$ is given by the Riemann--Roch theorem:
\[
  l_m=(\dim\g)(\deg\, D_m-g+1).
\]
Observe that
\[
  \deg D_m=-mN+m\sum_{i=1}^Ma_i+\sum_{i=1}^Mb_{m,i}+k|\Gamma|
\]
where $|\Gamma|$ denotes the number of elements in $\Gamma$.
By \refE{condd} we obtain
$
  \deg D_m=N+g-1+k|\Gamma|
$.
Hence
\[
 l_m=(\dim\g)(N+k|\Gamma|).
\]

The $\L_m$ is the subspace in $L(D_m)$ distinguished by the conditions $L_p\in\tilde\g_p$ ($p=-k,\ldots,k$) in \refE{ga_expan} given all over $\ga\in\Gamma$. At every $\ga\in\Gamma$, the codimension of the expansions of the form \refE{ga_expan}, given at $\ga$, in the space of all $\g$-valued power series in $z$ starting from $z^{-k}$ can be computed as
$c_\ga=\sum\limits_{p=-k}^{k-1} \codim_{\,\g}\,\tilde\g_p$ (starting from $p=k$ we have $\codim_{\,\g}\tilde\g_p=0$). The grading $\g=\bigoplus\limits_{p=-k}^k\g_p$ is symmetric in a sense that $\dim\g_p=\dim\g_{-p}$ \cite{Vin} (moreover, $\g_p$ and $\g_{-p}$ are contragredient as $\g_0$-modules). By definition $\codim_{\,\g}\tilde\g_p=\sum\limits_{q=p+1}^k\dim\g_q$, by symmetry $\sum\limits_{q=p+1}^k\dim\g_q=\dim\tilde\g_{-p-1}$, hence $\codim_{\,\g}\,\tilde\g_p+\codim_{\,\g}\tilde\g_{-p-1}=\dim\g$. Obviously, $c_\ga=\sum\limits_{p=-k}^{-1} (\codim_{\,\g}\,\tilde\g_p+\codim_{\,\g}\tilde\g_{-p-1})$.    Hence $c_\ga=k\dim\g$.

Further on, we have $\codim_{L(D_m)}\L_m=\sum\limits_{\ga\in\Gamma} c_\ga=k(\dim\g)|\Gamma|$. Finally\footnote{The idea of the following calculation can be traced back to \cite{KN_2point} (where it was given with no regard to the Lie algebra theory) via \cite{Sch_Laxmulti,Sh_DGr,KSlax}}
\[
 \dim\L_m=l_m-k(\dim\g)|\Gamma|=N\dim\g .
\]
\end{proof}



\section{Central extensions}
\label{S:cexst}
In this section we construct the almost graded central extensions of $\L$.  We call a central extension almost graded if it inherits the almost graded structure from the original Lie algebra while central elements are relegated to the degree $0$ subspace.

Almost graded central extensions are given by local cocycles.
Let us recall from \cite{KNFa,KSlax,SSlax,Sh_DGr,Sch_Laxmulti} that a two-cocycle $\eta$ on $\L$ is called local if $\exists M\in\Z_+$ such that for any $m,n\in\Z$, $|m+n|>M$, and any $L\in\L_m$, $L'\in\L_n$ we have $\eta(L,L')=0$.  Our main goal in this section is the following theorem.
\begin{theorem}\label{T:central}
\begin{itemize}
  \item[]
  \item[$1^\circ $] For any $L,L'\in\L$ the 1-form $\langle L,(d - \ad\w)L'\rangle$ is holomorphic except at the $P$- and $Q$-points where $\w$ is a $\g_0$-valued 1-form on $\Sigma$ having the expansion of the form
      \[
          \w(z)=\left(\frac{h}{z}+\w_0+\ldots\right)dz
      \]
      at any $\ga$-point, where $h\in\h$ is the element giving the grading on $\g$ at the point $\ga$.
  \item[$2^\circ$] For any invariant quadratic form $\langle\cdot , \cdot\rangle$ on $\g$ \[\eta(L,L')=\sum\limits_{i=1}^N\res_{P_i}\langle L,(d - \ad\w)L'\rangle\] gives a local cocycle on $\L$.
  \item[$3^\circ$] Up to equivalence, the almost-graded central extensions of $\L$ are in a one-to-one correspondence with the invariant quadratic forms on~$\g$. In particular, if $\g$ is simple then the central extension given by the cocycle $\eta$ is unique (in the class of the almost graded central extensions) up to equivalence and rescaling the central element.
\end{itemize}
\end{theorem}
\begin{proof}[Proof of \refT{central}{[\,$1^\circ$]}]
In a neighborhood of a $\ga$-point let
\[
   L=\sum_{p\ge -k} L_pz^p, \quad L'=\sum_{q\ge -k} L'_qz^q, \quad L_p\in{\tilde\g_p},\ L_q\in{\tilde\g_q}
\]
where $\tilde\g_p\subset\g$ is a filtration subspace, i.e. $\tilde\g_p=\bigoplus\limits_{s\le p}\g_s$, $\g_s$ being the grading subspaces. Then
\[
  dL'=\sum_{q\ge -k} qL'_qz^{q-1}\, dz,
\]
and in a neighborhood of the $\ga$-point
\begin{equation}\label{E:stand}
   \langle L,dL'\rangle=\sum_{p,q\ge -k} q\langle L_pL'_q\rangle z^{p+q-1}\, dz.
\end{equation}
Observe that $(\ad h)L_p=pL_p+\tilde L_{p-1}$ where $\tilde L_{p-1}\in\tilde\g_{p-1}$. For this reason
\[
\begin{split}
  (\ad\w)L' &=\ad(hz^{-1}+\w_0+\ldots)\sum_{q\ge -k} L'_qz^q\, dz= \\
  &=\sum_{q\ge -k}(qL'_q+\tilde L_{q-1})z^{q-1}\, dz +\sum_{\substack{q\ge -k\\l\ge 0\hphantom{ii}}} L_{q,l}z^{q+l}\, dz
\end{split}
\]
where $L_{q,l}=(\ad\w_l)L_q'\in\tilde\g_q$ (since $\w_l\in\g_0$). The terms $\tilde L_{q-1}z^{q-1}$ in the first sum have the same form as terms in the second sum (with $l=0$). By abuse of notation, we just omit them in the relation, regarding to them as to having been transferred to the second sum:
\begin{equation*}
 (\ad\w)L' = \sum_{q\ge -k} qL'_qz^{q-1}\, dz +\sum_{\substack{q\ge -k\\ l\ge 0\hphantom{ii}}} L_{q,l}z^{q+l}\, dz .
\end{equation*}
Hence
\begin{equation}\label{E:cobound}
 \langle L,(\ad\w)L'\rangle = \sum_{p,q\ge -k} q\langle L_p,L'_q\rangle z^{p+q-1}\, dz +\sum_{\substack{p,q\ge -k\\ l\ge 0\hphantom{i,ii}}}\langle L_p, L_{q,l}\rangle z^{p+q+l}\, dz ,
\end{equation}
Subtracting \refE{cobound} from \refE{stand} we obtain
\begin{equation}
 \langle L,(d-\ad\w)L'\rangle = -\sum_{\substack{p,q\ge -k\\ l\ge 0\hphantom{i,ii}}}\langle L_p, L_{q,l}\rangle z^{p+q+l}\, dz .
\end{equation}
Let us show that the last expression is holomorphic in the neighborhood of $z=0$. Assume, $p+q+l < 0$. Since $l\ge 0$, we obtain $p+q < 0$. By definition, $L_p\in\bigoplus\limits_{i\le p}\g_i$, $L_{q,l}\in\bigoplus\limits_{j\le q}\g_j$. Here, $i+j\le p+q<0$, hence $\langle \g_i,\g_j\rangle = 0$. It follows that $\langle L_p,L_{q,l}\rangle =0$.
\end{proof}
\begin{proof}[Proof  of {\refT{central}[\,$2^\circ$]}]
This part of the proof is nowadays standard \cite{KSlax,SSlax,Sh_DGr,Sh_G2}. We give it here for completeness.

As earlier, let $D_m$ be as follows:
\begin{equation}\label{E:Dm}
  D_m=-m\sum\limits_{i=1}^NP_i+\sum_{j=1}^M(a_jm+b_{m,j})\,Q_i+k\sum_{s=1}^K\ga_s
\end{equation}
Let the brackets $(\cdot)$ denote the divisor of what is closed inside them if the last is a function or a 1-form. Let $L\in\L_m$, $L'\in\L_{m'}$. Then
\[
  (\langle L,dL'\rangle)\ge (m+m'-1)\sum\limits_{i=1}^NP_i-\sum_{j=1}^M(a_j(m+m')+b_{m,j}+b_{m',j}-1)\,Q_j+D_\ga
\]
where $D_\ga$ is a certain divisor supported at $\ga$-points.

Assume that $(\w)\ge \sum\limits_{i=1}^Nm_i^+P_i-\sum_{j=1}^Mm_j^-Q_j-\sum_{s=1}^K\ga_s$. Then
\[
  (\langle L,(\ad\w)L'\rangle )\ge \sum\limits_{i=1}^N(m+m'+m_i^+)P_i-\sum_{j=1}^M(a_j(m+m')+b_{m,j}+b_{m',j}+m_j^-)\,Q_j+D'_\ga .
\]
where, by \refT{central}[\,$1^\circ$], $D_\ga-D'_\ga\ge 0$.

In order the 1-form $\rho=\langle L,(d-\ad\w)L'\rangle$ had a nontrivial residue at least at one of the points $P_i$ it is necessary that
\[
  \min_{i=1,\ldots,N} \{ m+m'-1,m+m'+m_i^+\}\le -1,
\]
in other words,
\begin{equation}\label{E:upperb}
  m+m'\le -1-\min_{i=1,\ldots,N} \{ -1,m_i^+\}.
\end{equation}

It is necessary also that $\sum\limits_{i=1}^N \res_{P_i}\rho\ne 0$. But then $\sum\limits_{j=1}^M \res_{Q_j}\rho\ne 0$ too, hence $\rho$ has a nontrivial residue at least at one of the points $Q_j$. For the last,  it is necessary that
\[
  \max_{j=1,\ldots,M} \{a_j(m+m')+b_{m,j}+b_{m',j}-1 ,a_j(m+m')+b_{m,j}+b_{m',j}+m_j^-\}\ge 1.
\]
Since $b_{j,m}\le B, \forall j,m$ (see page \pageref{bjmbound}) the last inequality implies that
\[
  \max_{j=1,\ldots,M} \{a_j(m+m')+2B-1 ,a_j(m+m')+2B+m_j^-\}\ge 1,
\]
further on
\[
  \max_{j=1,\ldots,M} \{a_j(m+m')+2B-1 ,a_j(m+m')+2B+\max_{j=1,\ldots,M}m_j^-\}\ge 1,
\]
then
\[
  \max_{j=1,\ldots,M} \{a_j(m+m')\}\ge 1-\max\{  2B-1\, ,\, 2B+\! \max_{j=1,\ldots,M}m_j^- \},
\]
and finally
\begin{equation}\label{E:lowerb}
   m+m' \ge \min_{j=1,\ldots,M}\{ a_j^{-1}(1-\max\{  2B-1\, ,\, 2B+\! \max_{j=1,\ldots,M}m_j^- \})\}.
\end{equation}
By \refE{upperb} and \refE{lowerb} we conclude that $\eta(L,L')$ is a local cocycle.
\end{proof}

\begin{remark}
Since $\ad\w$ is an inner derivation of $\g$, the part $\langle L,(\ad\w)L'\rangle$ of the cocycle $\eta$ is a coboundary. It can be explicitly represented as a linear function of the commutator $[L,L']$. Indeed, $\langle L,(\ad\w)L'\rangle=\tr (\ad L)(\ad[\w,L'])$. Making use of the relation $\ad[\w,L']=\ad\w\cdot\ad L'-\ad L'\cdot\ad\w$, and the cyclic property of the $\tr$ operator, we obtain $\langle L,(\ad\w)L'\rangle=-\tr(\ad[L,L']\cdot\ad\w)$. Thus, \refT{central}[\,$2^\circ$] can be reformulated as follows: the standard cocycle $\langle L,dL'\rangle$ is local up to a coboundary.
\end{remark}

\begin{proof}[Proof of \refT{central}{[\,$3^\circ$]}] For the proof of uniqueness in the case when $\g$ is simple we refer to \cite{SSlax} where this proof had been given for an arbitrary simple Lie algebra $\g$, and the 2-point case ($N=1$), and to \cite{Sch_Laxmulti} where it is given for arbitrary sets of $P$- and $Q$-points. The remainder of the statement $3^\circ$ of the Theorem easily follows from this result.
\end{proof}

\section{Relation to Tyurin parameters}\label{S:Tyupar}
In this section we consider gradings of the depth 1 and 2 on the classical Lie algebras, and of the depth 2 and 3 for $G_2$. In those cases we reproduce Tyurin parameters, and expansions of Lax operators earlier obtained in \cite{Klax,KSlax,Sh_DGr,Sh_G2}. We also point out the relevant new Lax operator algebras. We restrict ourselves with the gradings given by simple roots. For a simple root $\a_i$ such a grading is given by $h\in\h$ such that $\a_i(h)=1$, and $\a_j(h)=0$ ($j\ne i$). Hence the grading subspace $\g_p$ is a direct sum of the root subspaces $\g_\a$ such that $\a_i$ is contained in the expansion of $\a$ over simple roots with multiplicity $p$. Observe that the depth of grading given by a simple root is equal to the coefficient with which this root is contained in the expansion of the highest root. Below, we keep the following convention: $\g_i\subset\g$ is a subspace with the eigenvalue $(-i)$. At the pictures below the longest lines correspond to the medians of the corresponding matrices. We refer to \cite[Sect.2,\ \S 3.5]{Vin} for the detailed information on $\Z$-gradings of semisimple Lie algebras. Below, $e_1,\ldots,e_n$ denote the elements of a base in the space the root system is embedded in, orthonormal with respect to the Cartan--Killing form.

\subsection{The case of $A_n$}\label{S:An}

In this case $\g$ has $\left[\frac{n}{2}\right]$ gradings of depth 1 (and no gradings of depth 2 given by simple roots). The grading number $r$ ($1\le r\le\left[\frac{n}{2}\right]$) is given by assigning a coefficient $1$ to the simple root $\a_r=e_r-e_{r+1}$. Consider first the grading number $1$.
Then the block structure
\begin{figure}[h]           
\begin{picture}(0,0)
\put(165,57){\text{--}\ $\g_0$}
\put(165,74){\text{--}\ $\g_{-1}$}
\put(165,38){\text{--}\ $\g_{1}$}
\put(67,-10){\text{a}}
\end{picture}
  \includegraphics[width=6cm]{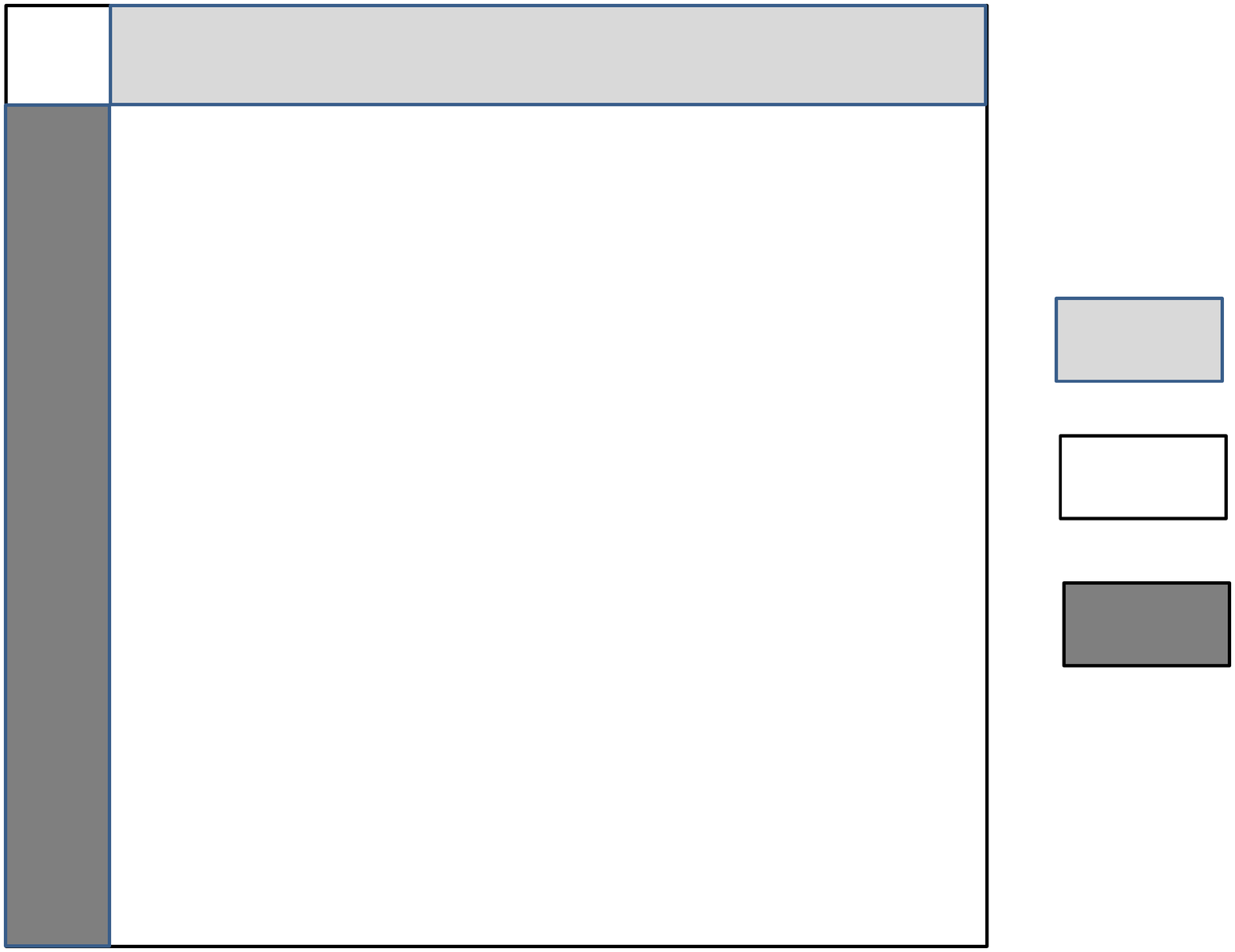}
\begin{picture}(30,0)
\end{picture}
\includegraphics[width=6cm]{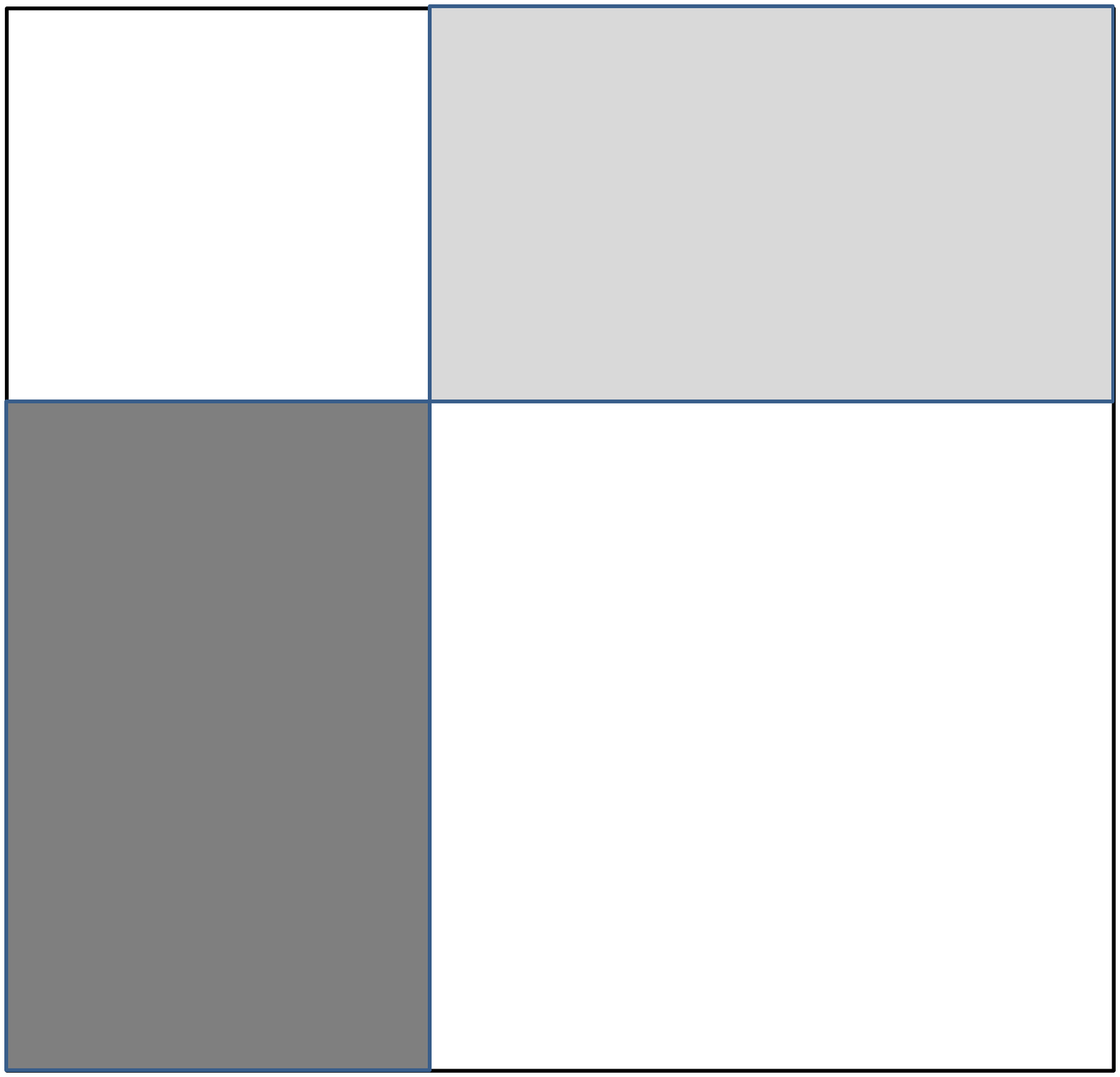}
\begin{picture}(0,0)
\put(-112,-10){\text{b}}
\end{picture}
  \caption{Case of $A_n$}\label{An}
\end{figure}
of the grading subspaces is as at the figure \ref{An},a. The matrices corresponding to the subspace $\g_{-1}=\tilde\g_{-1}$ can be represented as $\a\b^t$ where $\a$ is fixed by $\a^t=(1,0,\ldots,0)$, $\b\in\C^n$ is arbitrary.
Such matrix belongs to $\sln(n)$ if $\b^t\a=0$. Further on, an element $L_0\in\g$ belongs to the filtration subspace $\tilde\g_0=\g_{-1}\oplus\g_0$ if, and only if $\a$ is its eigenvector.
Thus we arrive to the following expansion of a Lax operator for $\sln(n)$ at a $\ga$-point \cite{Klax} (see also \cite{KSlax,Sh_DGr}):
\[
   L(z)=\a\b^tz^{-1}+L_0+\ldots
\]
where $\b^t\a=0$, and $L_0\a=\varkappa\a$ (for some $\varkappa\in\C$).

The grading number $r$ is given by the simple root $\a_r=e_r-e_{r+1}$. The corresponding matrix realization is presented at the figure \ref{An},b. This root is contained in the expansions of the roots $e_i-e_j$ for $i=1,\dots,r$, $j=r+1\ldots,n$. The sum of the corresponding root spaces gives the grading subspace $\g_{-1}$. Thus the upper $\g_0$-block has dimension $r\times r$, the $\g_{-1}$ and $\g_1$ blocks have dimensions $r\times (n-r)$, and $(n-r)\times r$, respectively. The subspace $\g_{-1}$ consists of the matrices of the form $\tilde\a_1\b_1^t+\ldots + \tilde\a_r\b_r^t$ where $\tilde\a_i=(0,\ldots,0,1,0,\ldots,0)$ ($1$ at the $i$-th position). $\tilde\a_i$ and $\b_j$ are mutually orthogonal, the linear span of $\tilde\a_1,\ldots,\tilde\a_r$ is an invariant subspace for $\tilde\g_0$.

\subsection{The case of $D_n$}\label{S:Dn}
The Dynkin diagram  $D_n$ is as follows:

\begin{figure}[h]  
\begin{picture}(100,45)
\put(-40,10){
\begin{picture}(100,30)
\put(0,10){\circle*{3}}
\put(0,10){\line(1,0){30}}
\put(30,10){\circle*{3}}
\put(30,10){\line(1,0){30}}
\put(60,10){\circle*{3}}
\put(60,10){\line(1,0){15}}
\put(100,10){\line(1,0){15}}
\put(115,10){\circle*{3}}
\put(115,10){\line(1,0){30}}
\put(145,10){\circle*{3}}
\put(145,10){\line(1,1){20}}
\put(165,30){\circle*{3}}
\put(145,10){\line(1,-1){20}}
\put(165,-10){\circle*{3}}
\put(80,9){$\ldots$}
\put(-5,0){$\a_1$}
\put(25,0){$\a_2$}
\put(150,7){$\a_{n-2}$}
\put(170,27){$\a_{n-1}$}
\put(170,-13){$\a_n$}
\end{picture}   }
\end{picture}
\end{figure}
where
\[
  \a_1=e_1-e_2,\ \ldots, \ \a_{n-1}=e_{n-1}-e_n,\ \a_n=e_{n-1}+e_n,
\]
and all positive roots are $e_i\pm e_j$, $1\le i<j\le n$.
Below we need expressions of the positive roots via simple ones. These are as follows:
\[
  e_i-e_j=\a_i+\ldots+\a_{j-1}\quad (1\le i<j\le n),
\]
and
\[
e_i+e_j=\left\{
   \begin{array}{ll}
            \a_i+\ldots+\a_{j-1}+2\a_j+\ldots+2\a_{n-2}+\a_{n-1}+\a_n, & \hbox{$i<j\le n-2$;} \\
            \a_i+\ldots+\a_{n-1}+\a_n, & \hbox{$i<j= n-1$;} \\
            \a_i+\ldots+\a_{n-2}+\a_n, & \hbox{$i\le n-2, j=n$;} \\
            \a_n, & \hbox{$i=n-1, j=n$.}
   \end{array}
        \right.
\]
The highest root is $\theta=\a_1+2\a_2+\ldots+2\a_{n-2}+\a_{n-1}+\a_n$, hence there are three simple roots $\a_1$, $\a_{n-1}$, and $\a_n$ giving gradings of depth $1$, but the last two are equivalent under an outer automorphism.

Consider the grading corresponding to the simple root $\a_1$. This simple root is contained in the expansions for $e_1\pm e_j$, $j=2,\ldots,n$. The corresponding root subspaces constitute the grading subspace $\g_{-1}$. In general, the blocks corresponding to the grading subspaces in the matrix realization of $\g=\so(2n)$ with respect to the quadratic form $\s=\begin{pmatrix}
          0 & E \\
          E & 0 \\
    \end{pmatrix}
$ are represented at figure \ref{Dn},a:
\begin{figure}[h]           
\begin{picture}(0,0)
\put(165,57){\text{--}\ $\g_0$}
\put(165,74){\text{--}\ $\g_{-1}$}
\put(165,38){\text{--}\ $\g_{1}$}
\put(67,-10){\text{a}}
\end{picture}
  \includegraphics[width=6cm]{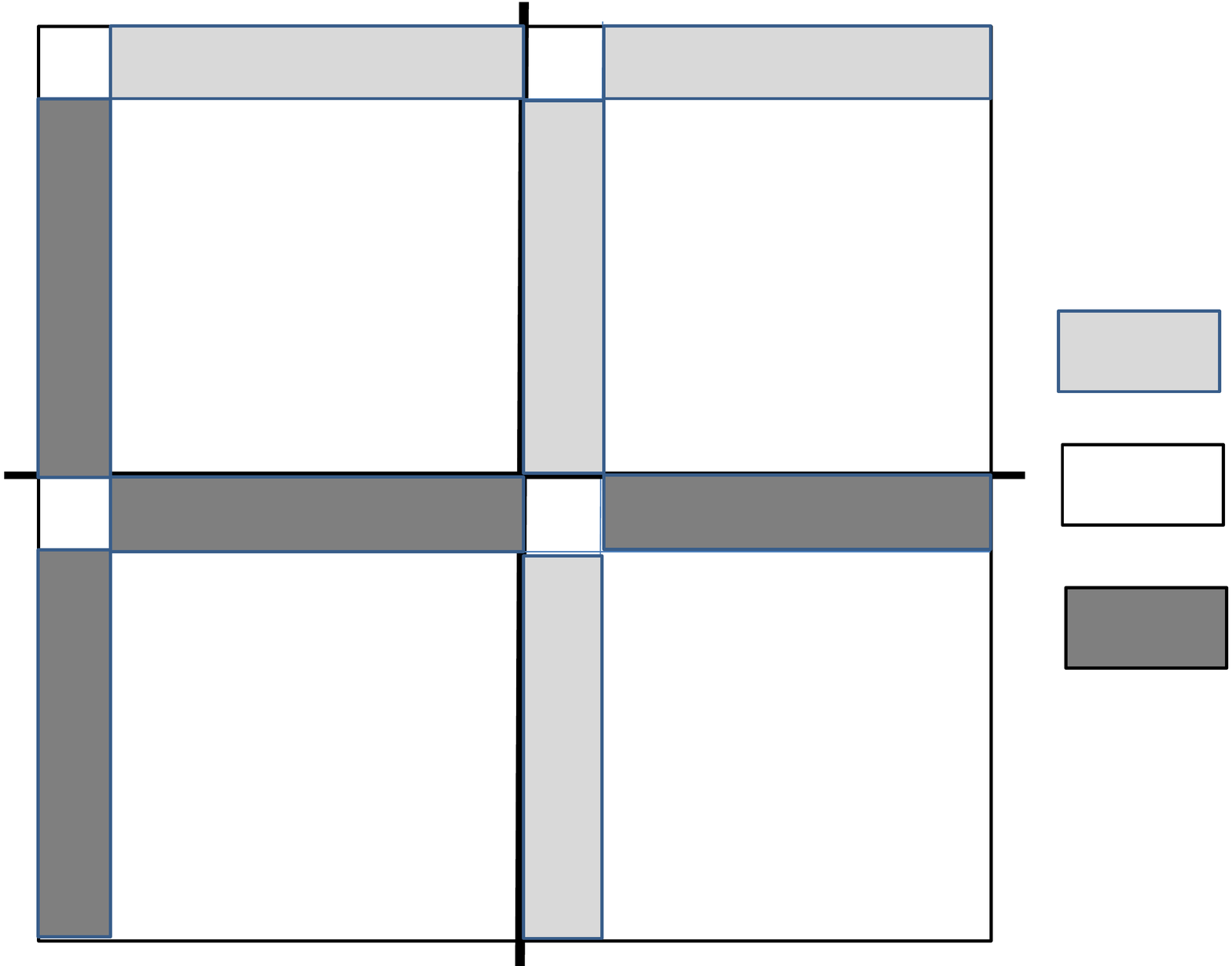}
\begin{picture}(30,0)
\put(-167,53){\small $0$}
\put(-105,109){\small $0$}
\end{picture}
\includegraphics[width=6cm]{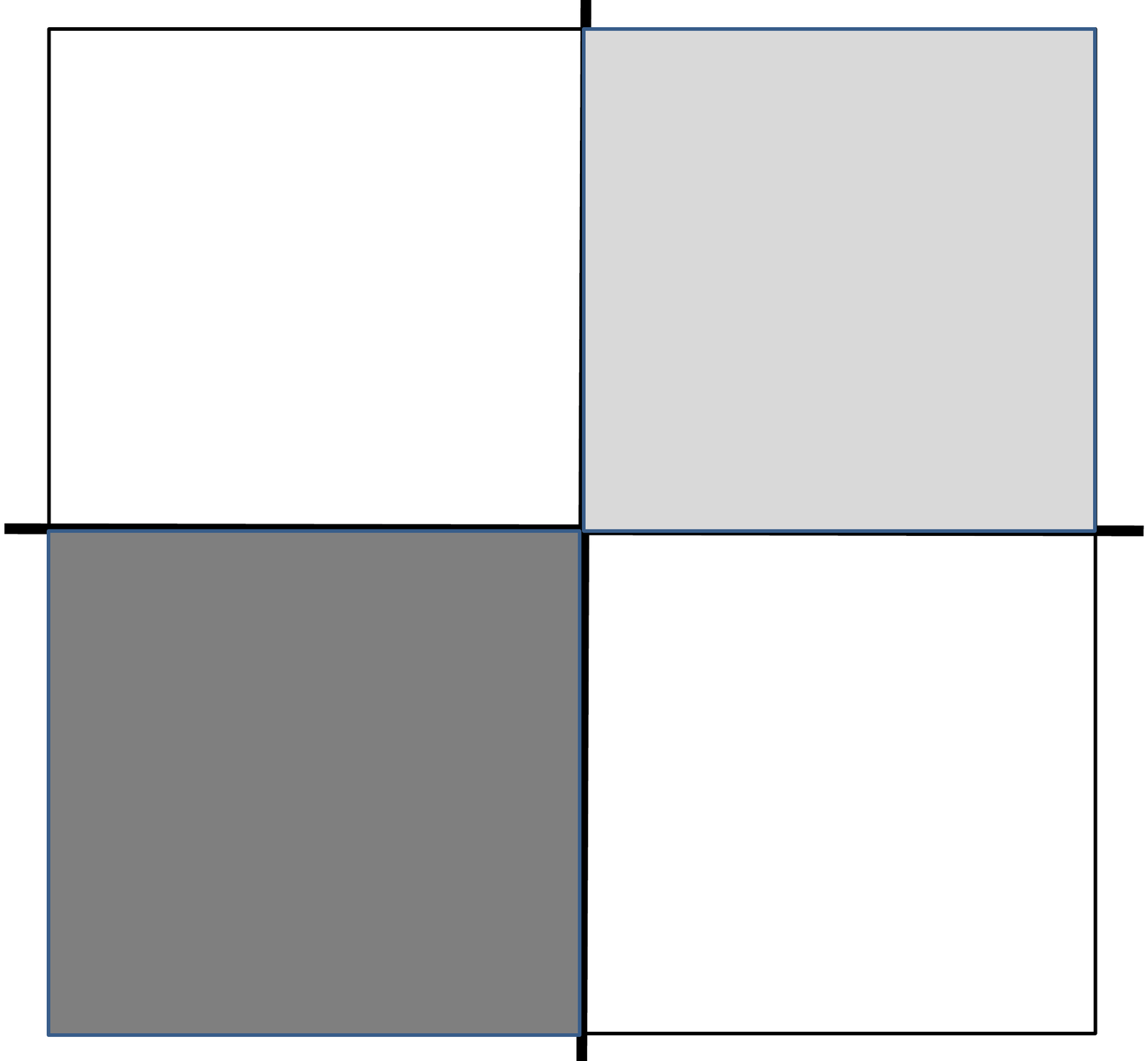}
\begin{picture}(0,0)
\put(-112,-10){\text{b}}
\end{picture}
  \caption{Case of $D_n$}\label{Dn}
\end{figure}

Observe that $\g_{-1}$ can be represented as the subspace of rank 2 matrices of the form $(\a\b^t-\b\a^t)\s$ where $\a,\b\in\C^{2n}$, $\a^t=(1,0,\ldots,0)$, and $\b^t\s\a=0$. Observe also that $\a$ is an eigenvector with respect to $\g_0$, and $\a^t\s\a=0$. Hence we obtain the following expansion (first found in \cite{KSlax}) for $L$ at a $\ga$-point :
\[
   L(z)=(\a\b^t-\b\a^t)\s z^{-1}+L_0+\ldots
\]
where $\a$, $\b$ and $L_0$ satisfy the above relations.

Considering now the grading given by the simple root $\a_n$ we see that it is contained in the expansions of the positive roots $e_i+e_j$ for all $i<j$. The remainder of simple roots forms the Dynkin diagram $A_{n-1}$. The blocks corresponding to the grading subspaces are represented at figure \ref{Dn},b. In particular, $\g_{-1}$ is represented by the matrices of the form $(\a_1\b_1^t-\b_1\a_1^t)\s+\ldots+(\a_n\b_n^t-\b_n\a_n^t)\s$ where $\a_i,\b_i\in\C^{2n}$, for $i=1,\ldots,n$ $\a_i$ is given by its coordinates $\a_i^j=\d_i^j$ (where $j=1,\ldots,2n$, $\d_i^j$ is the Kronecker symbol), $\b_i=(\b_i^1,\ldots,\b_i^n,0,\ldots,0)$ (where $\b_i^j\in\C$ are arbitrary). Observe that $\a_i^t\s\b_j=\a_i^t\s\a_j=\b_i^t\s\b_j=0$ for all $i,j=1,\ldots,n$. Using these relations we can independently show by methods of \cite{KSlax} that the property of $L$ to have simple poles at $\ga$-points is conserved under commutator.
\subsection{The case of $C_n$}\label{S:Cn}
The Dynkin diagram  $C_n$ is as follows:
\begin{figure}[h]
\begin{picture}(100,45)
\put(-40,10){
\begin{picture}(100,30)
\put(0,10){\circle*{3}}
\put(0,10){\line(1,0){30}}
\put(30,10){\circle*{3}}
\put(30,10){\line(1,0){30}}
\put(60,10){\circle*{3}}
\put(60,10){\line(1,0){15}}
\put(100,10){\line(1,0){15}}
\put(115,10){\circle*{3}}
\put(115,10){\line(1,0){30}}
\put(145,10){\circle*{3}}
\put(145,11){\line(1,0){30}}
\put(175,10){\circle*{3}}
\put(145,9){\line(1,0){30}}
\put(175,10){\line(-2,1){10}}
\put(175,10){\line(-2,-1){10}}
\put(80,9){$\ldots$}
\put(-5,0){$\a_1$}
\put(25,0){$\a_2$}
\put(130,0){$\a_{n-1}$}
\put(170,0){$\a_n$}
\end{picture}   }
\end{picture}
\end{figure}
\newline where
\[
  \a_1=e_1-e_2,\ \ldots, \ \a_{n-1}=e_{n-1}-e_n,\ \a_n=2e_n,
\]
and all positive roots are $e_i\pm e_j$, $1\le i<j\le n$, and $2e_i$ ($i=1,\ldots,n$).
The expressions of the positive roots via simple ones are as follows:
\[
   \begin{array}{lll}
       e_i-e_j&=\a_i+\ldots+\a_{j-1}, & \hbox{$1\le i<j\le n$;} \\
       2e_i&=2\a_i+\ldots\hphantom{+\a_{j-1}}+2\a_{n-1}+\a_n, & \hbox{$i=1,\ldots,n-1$;}\\
       2e_n&=\hphantom{2\a_i+\ldots+\a_{j-1}+2\a_{n-1}+\ }\a_n,
   \end{array}
\]
and
\[
e_i+e_j=\left\{
   \begin{array}{ll}
            \a_i+\ldots+\a_{j-1}+2\a_j+\ldots+2\a_{n-1}+\a_n, & \hbox{$i<j\le n-1$;} \\
            \a_i+\ldots+\a_{n-1}+\a_n, & \hbox{$i<j= n$;}.
   \end{array}
        \right.
\]
The highest root is $\theta=2e_1=2\a_1+2\a_2+\ldots+2\a_{n-1}+\a_n$. We will consider here the grading of depth $2$ given by $\a_1$, and the grading of depth $1$ given by $\a_n$. The Lax operator algebra corresponding to the first of them is found in \cite{KSlax}, see also \cite{Sh_DGr}. The algebra corresponding to the second grading is new.

The blocks corresponding to the grading subspaces in the matrix realization of $\g=\spn(2n)$ with respect to the symplectic form $\s=\begin{pmatrix}
          0 & E \\
          -E & 0 \\
    \end{pmatrix}
$ are represented at figure \ref{Cn},a.
\begin{figure}[h]          
\begin{picture}(0,0)
\put(165,92){\text{--}\ $\g_{-2}$}
\put(165,21){\text{--}\ $\g_2$}
\put(165,57){\text{--}\ $\g_0$}
\put(165,74){\text{--}\ $\g_{-1}$}
\put(165,38){\text{--}\ $\g_{1}$}
\put(67,-10){\text{a}}
\end{picture}
  \includegraphics[width=6cm]{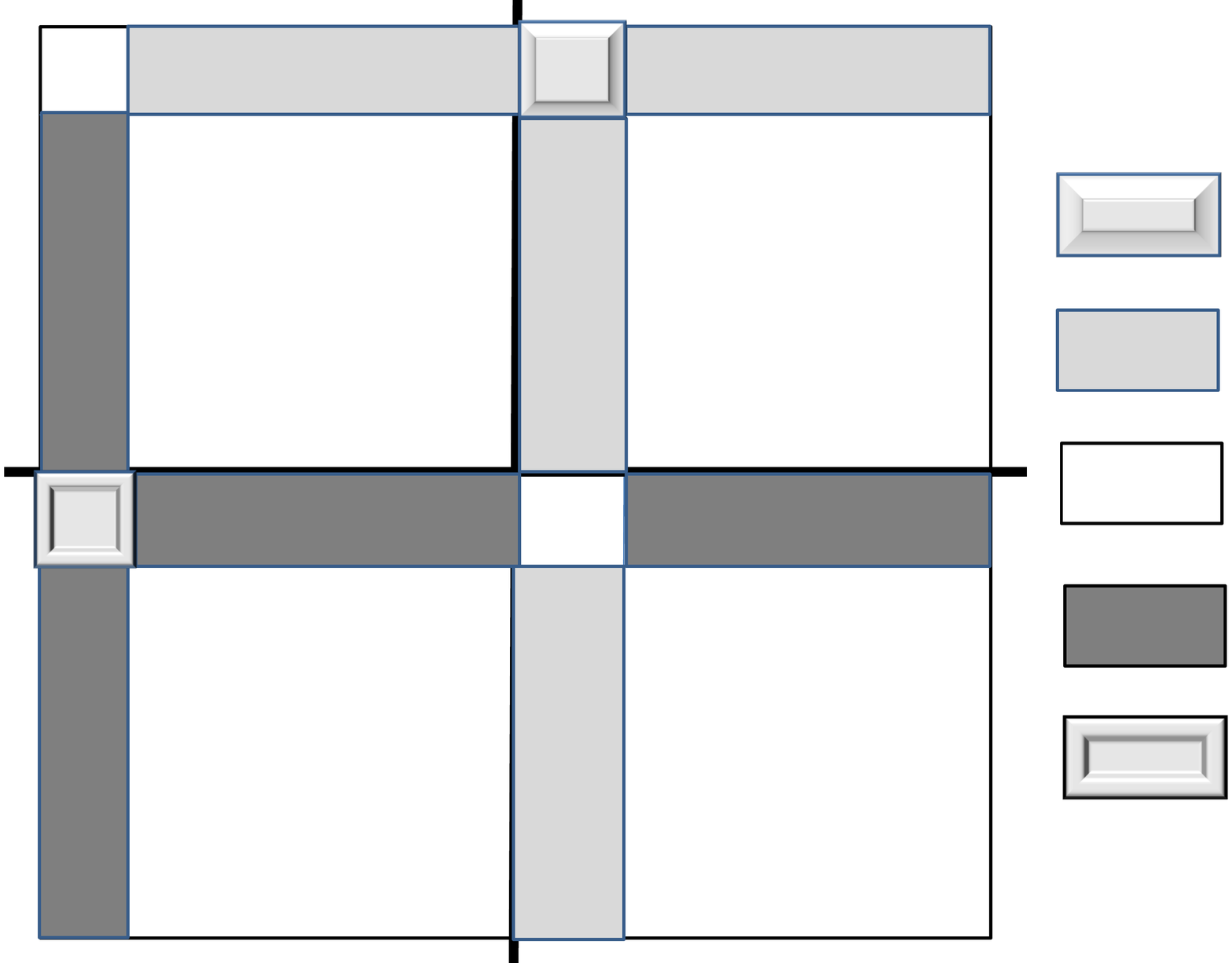}
\begin{picture}(30,0)
\end{picture}
\includegraphics[width=6cm]{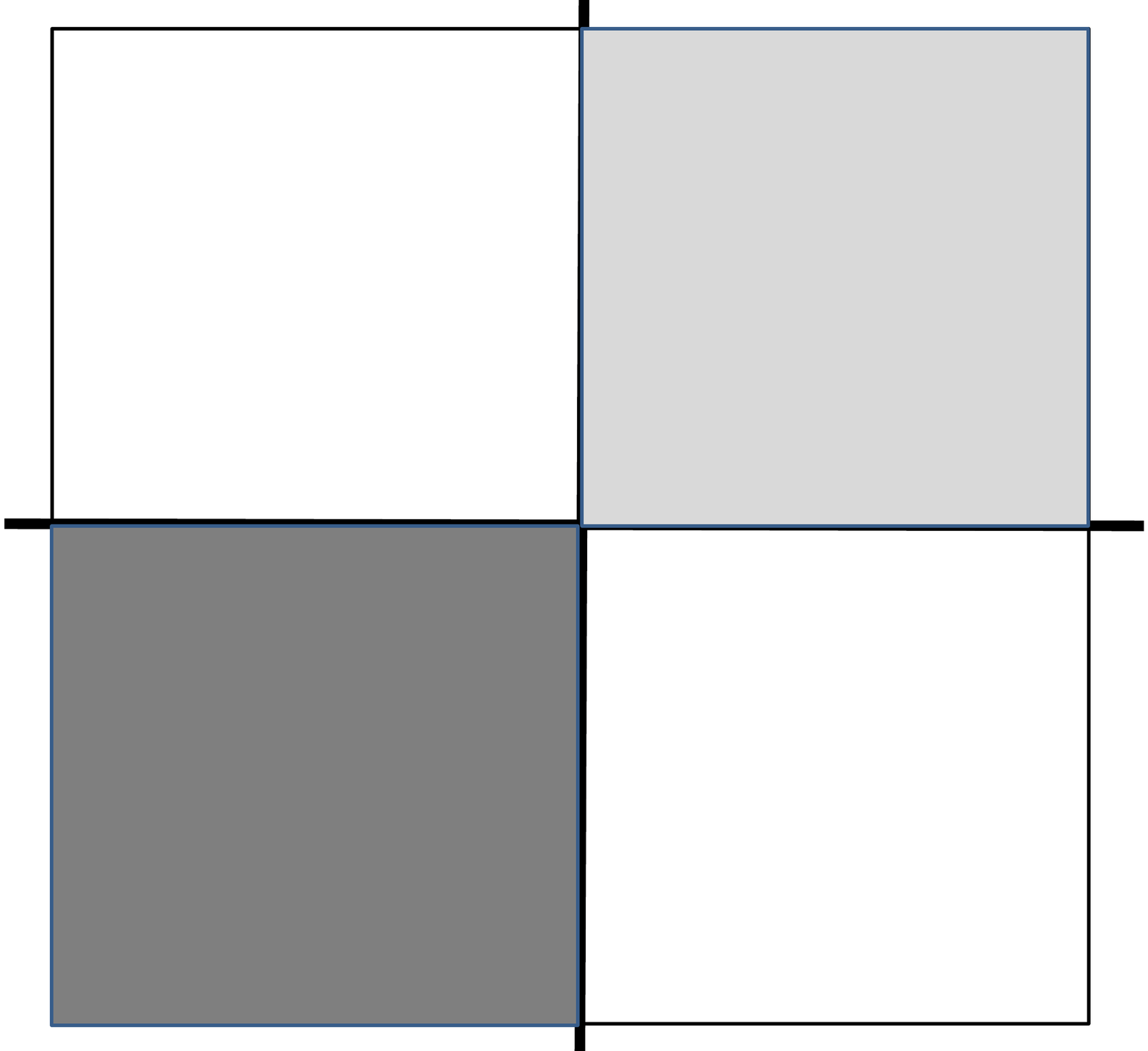}
\begin{picture}(0,0)
\put(-112,-10){\text{b}}
\end{picture}
  \caption{Case of $C_n$}\label{Cn}
\end{figure}
In particular, the one-dimensional subspace $\g_{-2}$ corresponding to the highest root consists of matrices of the form $\nu\a\a^t\s$ where $\a\in\C^{2n}$, $\a^t=(1,0,\ldots,0)$, $\nu\in\C$. The subspace $\g_{-1}$ consists of matrices of the form $(\a\b^t+\b\a^t)\s$ where $\a,\b\in\C^{2n}$, $\a^t=(1,0,\ldots,0)$, and $\b^t\s\a=0$. Observe also that $\a$ is an eigenvector with respect to $\g_0$. Finally, observe that for every $L_1\in\g_1$ we have $\a^t\s L_1\a=0$. Hence we have arrived to the following form of expansions of the element $L$ at $\ga$-points:
\[
   L(z)=\nu\a\a^t\s z^{-2}+(\a\b^t+\b\a^t)\s z^{-1}+L_0+L_1z+\ldots
\]
where $\a$, $\b$, $L_0$, $L_1$ satisfy the just listed relations (see also \cite{KSlax,Sh_DGr}).

The matrix realization of the grading given by a simple root $\a_n$ is represented at the figure \ref{Cn},b. It is completely similar to that for the system $D_n$, with the only distinction that matrices in $\g_{-1}$ are of the form $(\a_1\b_1^t+\b_1\a_1^t)\s+\ldots+(\a_n\b_n^t+\b_n\a_n^t)\s$.
Again, $\a_i$ and $\b_j$ possess the orthogonality relations enabling one independently show by methods of \cite{KSlax} that simple poles at $\ga$-points remain simple under commutator (see \refS{Dn}).


\subsection{The case of $B_n$}\label{S:Bn}
The Dynkin diagram  $B_n$ is as follows:
\begin{figure}[h]
\begin{picture}(100,45)
\put(-40,10){
\begin{picture}(100,30)
\put(0,10){\circle*{3}}
\put(0,10){\line(1,0){30}}
\put(30,10){\circle*{3}}
\put(30,10){\line(1,0){30}}
\put(60,10){\circle*{3}}
\put(60,10){\line(1,0){15}}
\put(100,10){\line(1,0){15}}
\put(115,10){\circle*{3}}
\put(115,10){\line(1,0){30}}
\put(145,10){\circle*{3}}
\put(145,11){\line(1,0){30}}
\put(175,10){\circle*{3}}
\put(145,9){\line(1,0){30}}
\put(145,10){\line(2,1){10}}
\put(145,10){\line(2,-1){10}}
\put(80,9){$\ldots$}
\put(-5,0){$\a_1$}
\put(25,0){$\a_2$}
\put(133,0){$\a_{n-1}$}
\put(170,0){$\a_n$}
\end{picture}   }
\end{picture}
\end{figure}
\newline where
\[
  \a_1=e_1-e_2,\ \ldots, \ \a_{n-1}=e_{n-1}-e_n,\ \a_n=e_n,
\]
and all positive roots are $e_i\pm e_j$, $1\le i<j\le n$, and $e_i$ ($i=1,\ldots,n$).
The expressions of the positive roots via simple ones are as follows:
\[
   \begin{array}{lll}
       e_i-e_j&=\a_i+\ldots+\a_{j-1}, &  \\
       e_i&=\a_i+\ldots\hphantom{+\a_{j-1}}+\a_{n-1}+\a_n, & \hbox{$1\le i<j\le n$;}\\
       e_i+e_j&=\a_i+\ldots+\a_{j-1}+2\a_j+\ldots+2\a_{n-1}+2\a_n,&
   \end{array}
\]

The highest root is $\theta=e_1+e_2=\a_1+2\a_2+\ldots+2\a_n$. We will consider here the grading of depth $1$ given by $\a_1$, and the grading of depth $2$ given by $\a_n$. The Lax operator algebra corresponding to the first of them is found in \cite{KSlax}, see also \cite{Sh_DGr}. The algebra corresponding to the second grading is new.

The blocks corresponding to the grading subspaces in the matrix realization of $\g=\so(2n+1)$ with respect to the quadratic form $\s~=~\begin{pmatrix}
          0 & 0 & E \\
          0 & 1 & 0 \\
          E & 0 & 0 \\
    \end{pmatrix}
$ are represented at figure \ref{Bn},a.
\begin{figure}[h] 
\begin{picture}(0,0)
\put(165,92){\text{--}\ $\g_{-2}$}
\put(165,21){\text{--}\ $\g_2$}
\put(165,57){\text{--}\ $\g_0$}
\put(165,74){\text{--}\ $\g_{-1}$}
\put(165,38){\text{--}\ $\g_{1}$}
\put(67,-10){\text{a}}
\end{picture}
  \includegraphics[width=6cm]{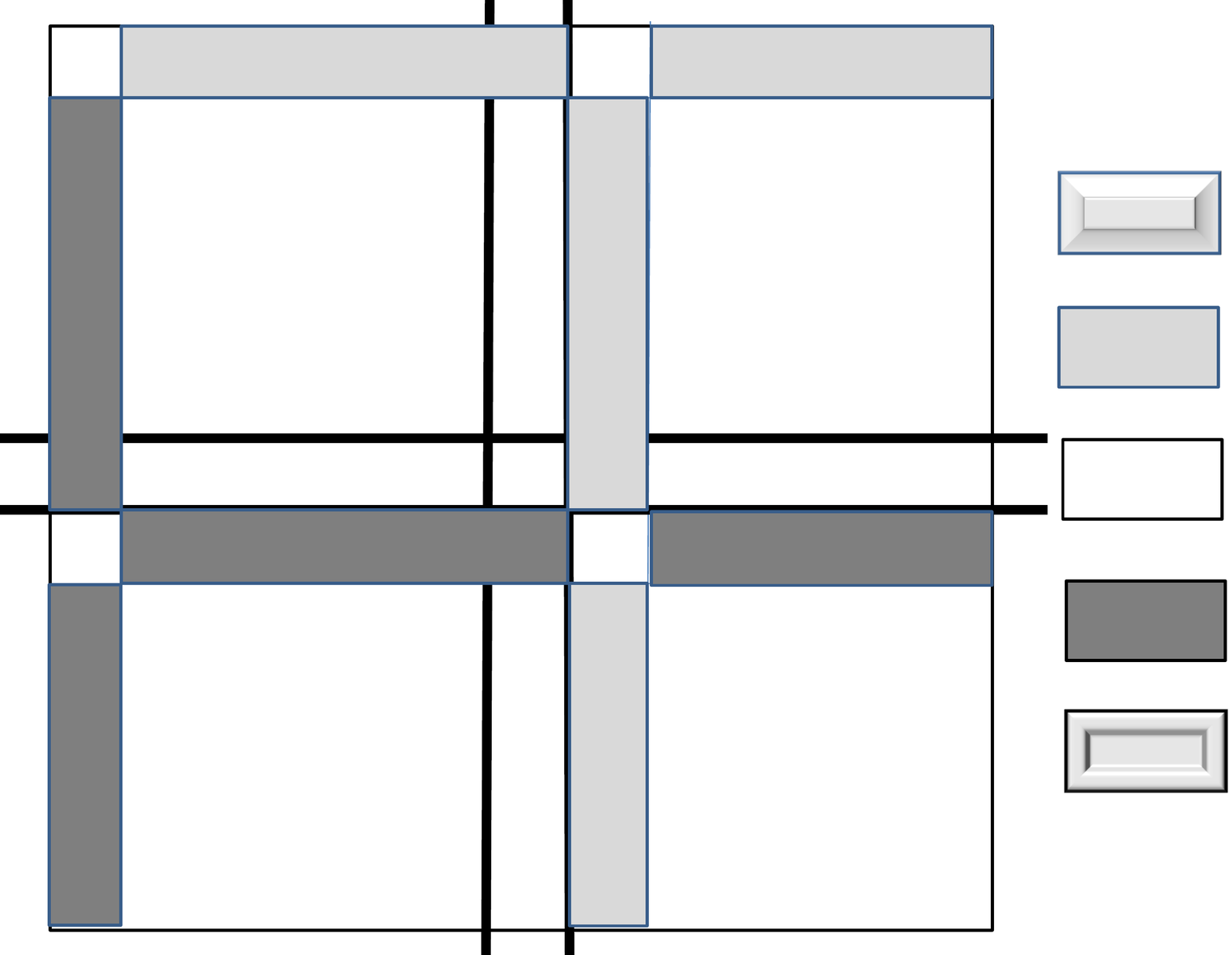}
\begin{picture}(30,0)
\put(-101,109){\small $0$}
\put(-101,48){\small $0$}
\put(-111,57){\small $0$}
\put(-167,48){\small $0$}
\end{picture}
\includegraphics[width=6cm]{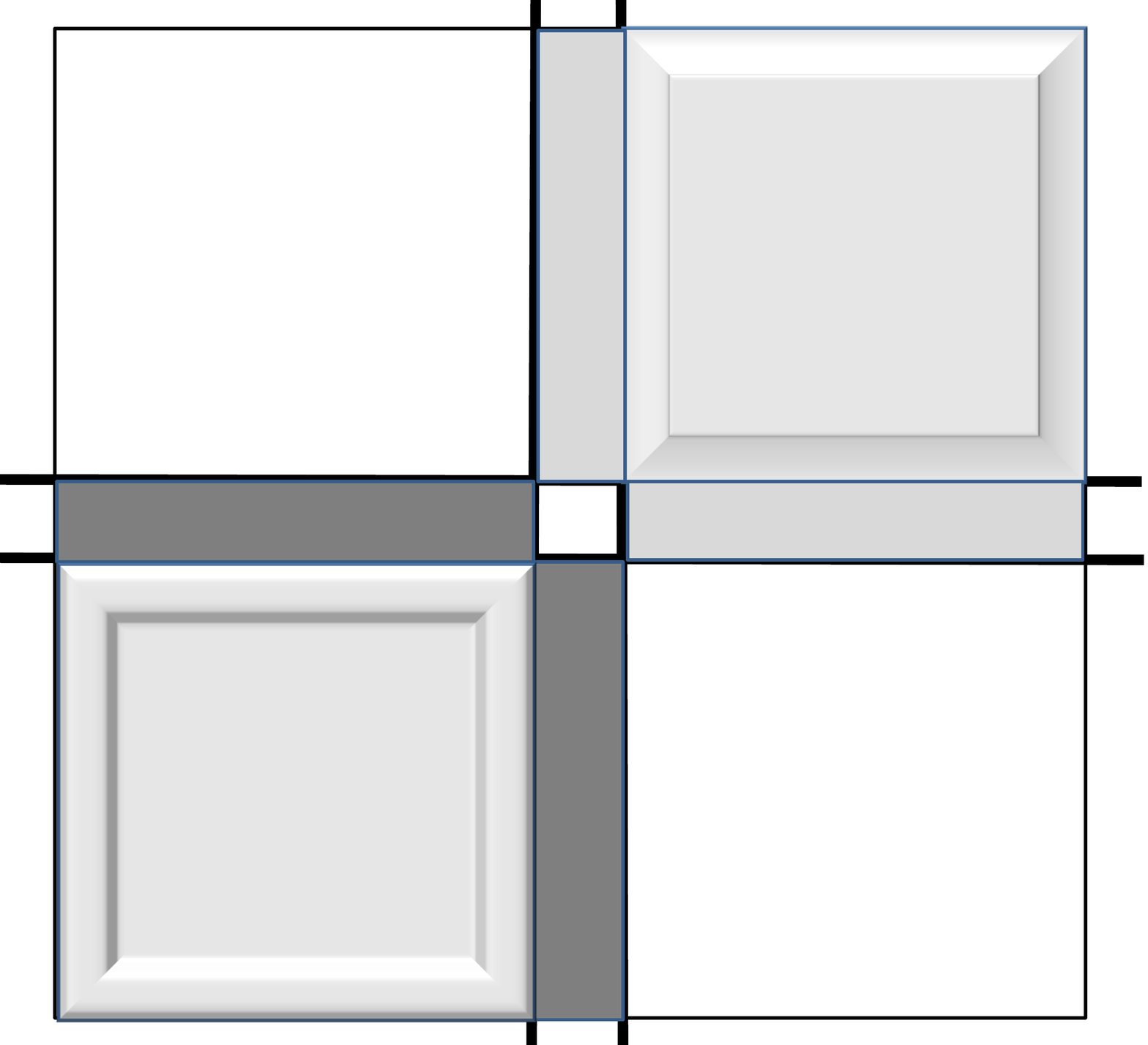}
\begin{picture}(0,0)
\put(-111,57){\small $0$}
\put(-112,-10){\text{b}}
\end{picture}
  \caption{Case of $B_n$}\label{Bn}
\end{figure}

The subspace $\g_{-1}$ consists of matrices of the form $(\a\b^t-\b\a^t)\s$ where $\a,\b\in\C^{2n}$, $\a^t=(1,0,\ldots,0)$, and $\b^t\s\a=0$ (observe that also $\a^t\s\a=0$). Again, $\a$ is an eigenvector with respect to $\g_0$. Hence we have arrived to the following form of expansions of the element $L$ at $\ga$-points:
\[
   L(z)=(\a\b^t-\b\a^t)\s z^{-1}+L_0+\ldots
\]
where $\a$, $\b$, $L_0$ satisfy the just listed relations (see also \cite{KSlax,Sh_DGr}).

The matrix realization of the grading given by a simple root $\a_n$ is represented at the figure \ref{Bn},b. The subspace $\g_{-1}$ is a direct sum of the root subspaces corresponding to the roots $e_i$ ($i=1,\ldots,n$). The subspace $\g_{-2}$ is a direct sum of the root subspaces corresponding to the roots $e_i+e_j$ for all $i,j=1,\ldots,n$. Matrices in $\g_{-1}$ are of the form $\a_0\b_0^t-\b_0\a_0^t$ where $\a_0,\b_0\in\C^{2n+1}$, $(\a_0)_i=(\d_{i,n+1})$, $\b_0=(\b_0^1,\ldots,\b_0^n,0,\ldots,0)$ ($\b_0^j\in\C$ are arbitrary). Matrices in $\g_{-2}$ are of the form $(\a_1\b_1^t-\b_1\a_1^t)\s+\ldots+(\a_n\b_n^t-\b_n\a_n^t)\s$
where $\a_i,\b_i\in\C^{2n+1}$, for $i=1,\ldots,n$, $\a_i$ is given by its coordinates $\a_i^j=\d_i^j$ (where $j=1,\ldots,2n+1$, $\d_i^j$ is the Kronecker symbol), $\b_i=(\b_i^1,\ldots,\b_i^n,0,\ldots,0)$ (where $\b_i^j\in\C$ are arbitrary).

\subsection{The case of $G_2$}\label{S:G2}
The Dynkin diagram  $G_2$ is as follows:
\begin{figure}[h]
\begin{picture}(100,45)
\put(-40,10){
\begin{picture}(100,30)
\put(0,10){\circle*{4}}
\put(0,10){\line(1,0){30}}
\put(0,12){\line(1,0){30}}
\put(0,8){\line(1,0){30}}
\put(30,10){\circle*{4}}
\put(20,10){\line(-2,-1){10}}
\put(20,10){\line(-2,1){10}}

\put(-5,0){$\a_1$}
\put(25,0){$\a_2$}
\end{picture}   }
\end{picture}
\end{figure}
\newline where $\a_1,\a_2\in\C^2$,
\[
  \a_1=(1,0),\  \ \a_2=(-3/2,\sqrt{3}/2),
\]
and all positive roots are vertices of two regular hexagons determined by $\a_1$, $\a_2$ (with a common center at $(0,0)$). The Lie algebra $G_2$ has an exect $7$-dimensional representation by matrices of the form are presented at the figure \ref{G2},b. At the figure, dependent blocks are of the same color (bright gray, dark gray, or wight). By $[x]$ (where $x\in\C^3$, $x^T=(x_1,x_2,x_3)$) we denote a skew-symmetric matrix given by $[x]=\bigl(\begin{smallmatrix} 0&x_3&-x_2\\
-x_3&0&x_1\\
x_2&-x_1&0\end{smallmatrix}\bigr)$. Below we give the full list of positive roots with the relation to the matrix elements of the representation. For every positive root we point out one independent entry of submatrices $a_1$, $a_2$, $A$ giving this root, and the list of all corresponding non-vanishing entries of the full matrix (in the form $(i,j)$).
\[
   \begin{array}{lll}
       \a_1, & (a_1)_1=\sqrt{2}, & (2,1),(1,5),(6,4), \\
       \a_2, &  A_{21}=1,& (3,2),(5,6),\\
       \a_1+\a_2, & (a_1)_2=\sqrt{2},  &  (3,1),(1,6),(5,4),  \\
      2\a_1+\a_2, & (a_2)_3=\sqrt{2},   &  (7,1),(1,4),(2,6),  \\
      3\a_1+\a_2, & A_{13}=1,  &  (2,4),(7,5),  \\
     3\a_1+2\a_2, & A_{23}=1,  &  (3,4),(7,6).
   \end{array}
\]

The highest root is $\theta=3\a_1+2\a_2$. Let us consider first the grading of depth $2$ given by $\a_2$. The blocks corresponding to the grading subspaces in the matrix realization are represented at figure \ref{G2},a.
\begin{figure}[h]  
\begin{picture}(0,0)
\put(165,92){\text{--}\ $\g_{-2}$}
\put(165,21){\text{--}\ $\g_2$}
\put(165,57){\text{--}\ $\g_0$}
\put(165,74){\text{--}\ $\g_{-1}$}
\put(165,38){\text{--}\ $\g_{1}$}
\put(67,-10){\text{a}}
\end{picture}
 \includegraphics[width=6cm]{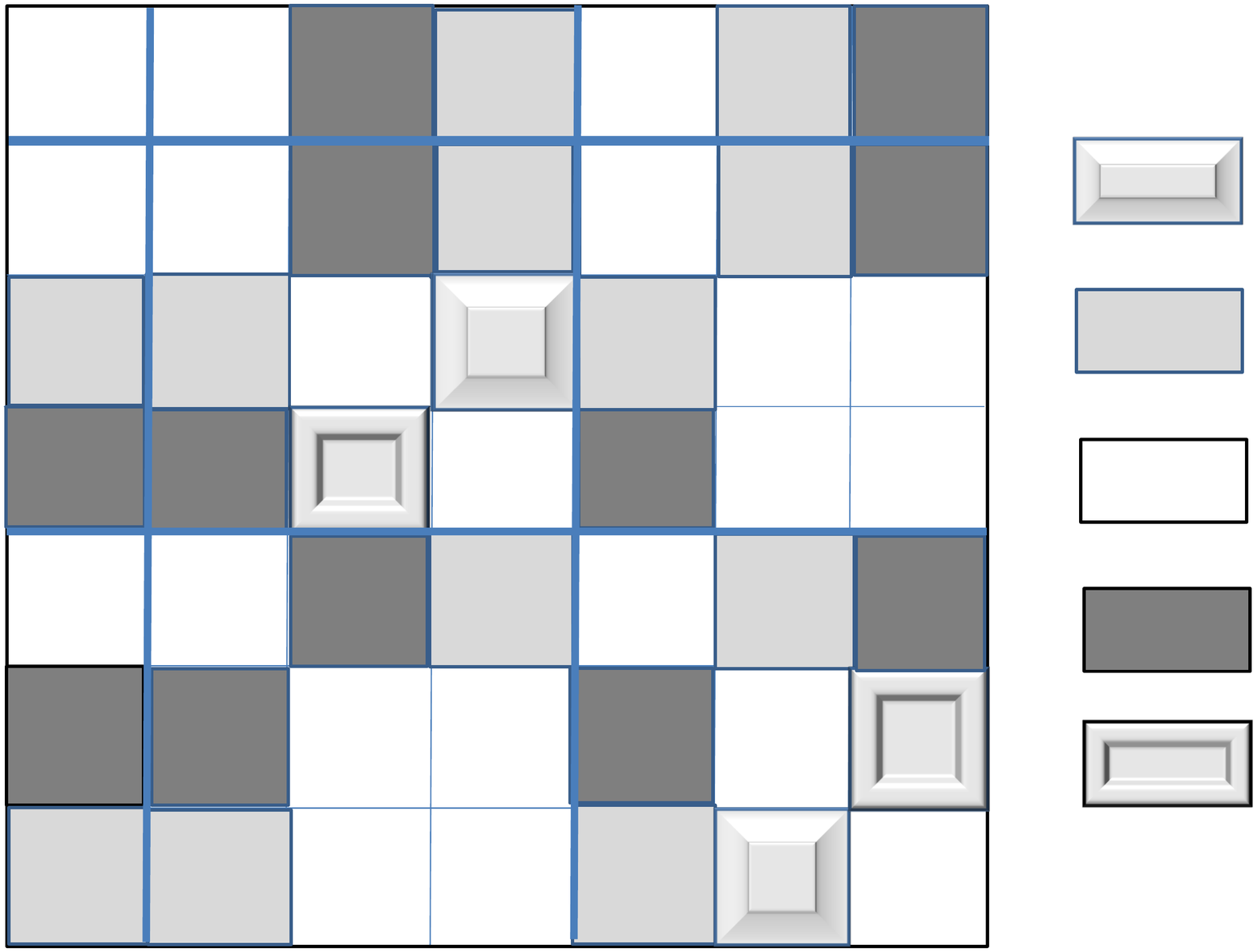}
\begin{picture}(30,0)
\put(-163,105){\small $0$}
\put(-95,90){\small $0$}
\put(-79,75){\small $0$}
\put(-60,57){\small $0$}
\put(-147,42){\small $0$}
\put(-130,25){\small $0$}
\put(-112,8){\small $0$}
\put(-105,-10){\text{a}}
\end{picture}
\includegraphics[width=6cm]{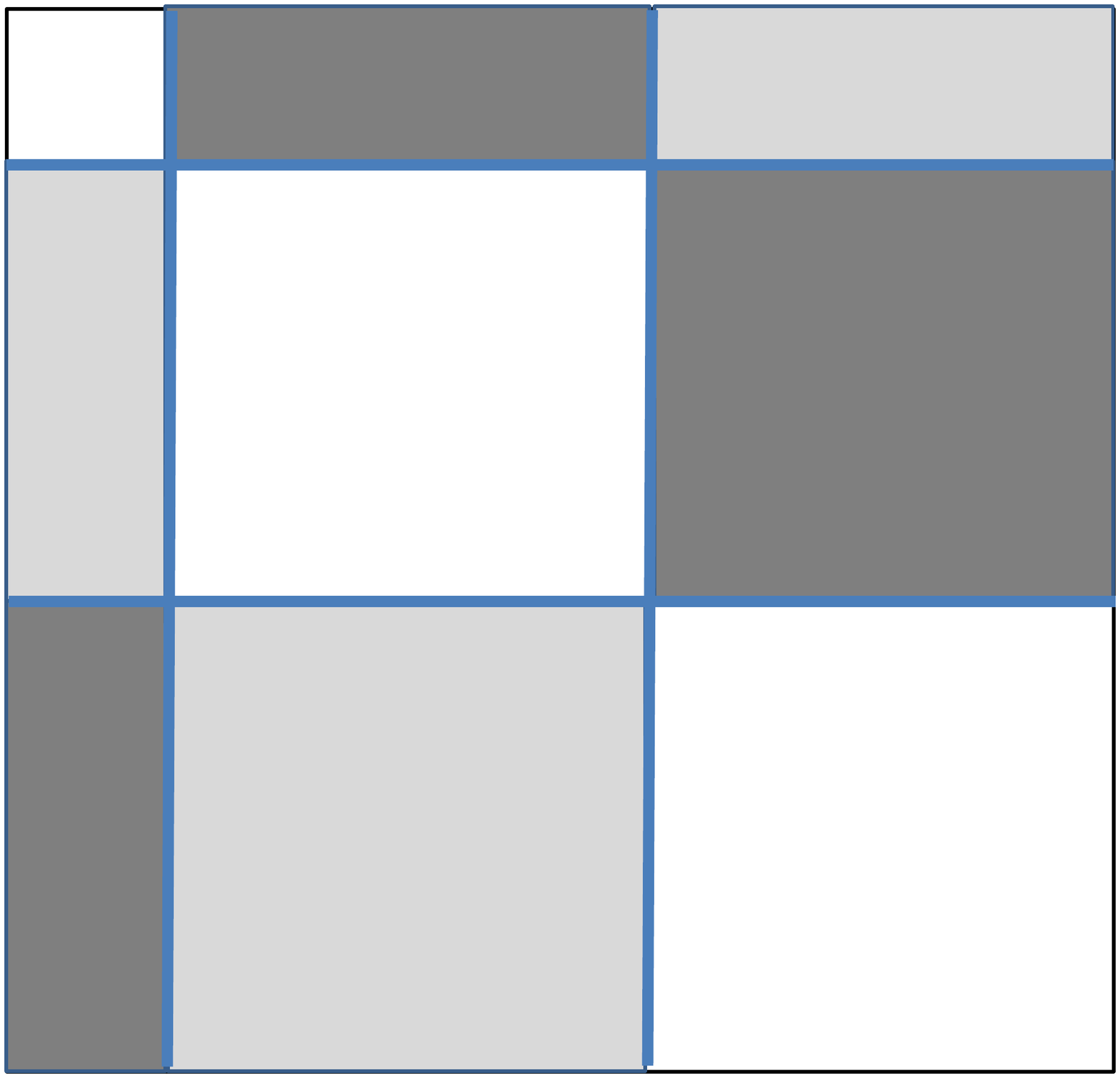}
\begin{picture}(0,0)
\put(-163,105){\small $0$}
\put(-80,105){\small $-a_1^T$}
\put(-165,75){\small $a_1$}
\put(-165,25){\small $a_2$}
\put(-130,105){\small $-a_2^T$}
\put(-130,75){$A$}
\put(-85,25){$-A^T$}
\put(-140,25){$\frac{1}{\sqrt{2}}[a_1]$}
\put(-90,75){$\frac{1}{\sqrt{2}}[a_2]$}
\put(-112,-10){\text{b}}
\end{picture}
\caption{Case of $G_2$: depth $2$}\label{G2}
\end{figure}

It is easy to check that the subspace $\g_{-2}$ consists of matrices of the form
\begin{equation}\label{E:L{-2}}
   L_{-2}=\mu\begin{pmatrix}
       0 & 0 & 0\\
       0 & \tilde\a_1\tilde\a_2^t &  0   \\
       0 &  0  & -\tilde\a_2\tilde\a_1^t    \\
       \end{pmatrix},\quad \mu\in\C
\end{equation}
where $\tilde\a_1=(0,1,0)$, $\tilde\a_2=(0,0,1)$ while
the subspace $\g_{-1}$ consists of matrices of the form
\begin{equation}\label{E:resid}
   L_{-1}=\begin{pmatrix}
       0 & -\sqrt{2}\b_{02}\tilde\a_2^t & -\sqrt{2}\b_{01}\tilde\a_1^t \\
       \sqrt{2}\b_{01}\tilde\a_1 & \tilde\a_1\b_2^t-\b_1\tilde\a_2^t & \b_{02}[\tilde\a_2] \\
       \sqrt{2}\b_{02}\tilde\a_2 & \b_{01}[\tilde\a_1] & \tilde\a_2\b_1^t-\b_2\tilde\a_1^t \\
       \end{pmatrix}.
\end{equation}
where $\b_{01},\b_{02}\in\C$ are arbitrary, $\b_1,\b_2\in \C^3$ satisfy the following orthogonality relations: $\tilde\a_1^T\b_2=0$, $\tilde\a_2^T\b_1=0$. Observe also that $\tilde\a_1^T\tilde\a_2=0$
and if $L_0\in\tilde\g_0$ is given as at the figure \ref{G2},b then
\begin{equation}\label{E:eigen}
\tilde\a_1^Ta_2=0, \quad \tilde\a_2^Ta_1=0, \quad A\tilde\alpha_1=\varkappa_1\tilde\alpha_1,  \quad -A^T\tilde\alpha_2=\varkappa_2\tilde\alpha_2
\end{equation}
where $\varkappa_1,\varkappa_2\in\C$.

As a result we have obtained the Lax operator algebra recently found in \cite{Sh_G2}. Hence we may claim that this Lax operator algebra corresponds to the depth $2$ grading of $G_2$ given by the simple root $\a_2$.

Besides, the Lie algebra $G_2$ has a grading of depth $3$ given by the simple root $\a_1$. The matrix realization of this grading is given at the figure \ref{G2c}.
\begin{figure}[h] 
\includegraphics[width=6cm]{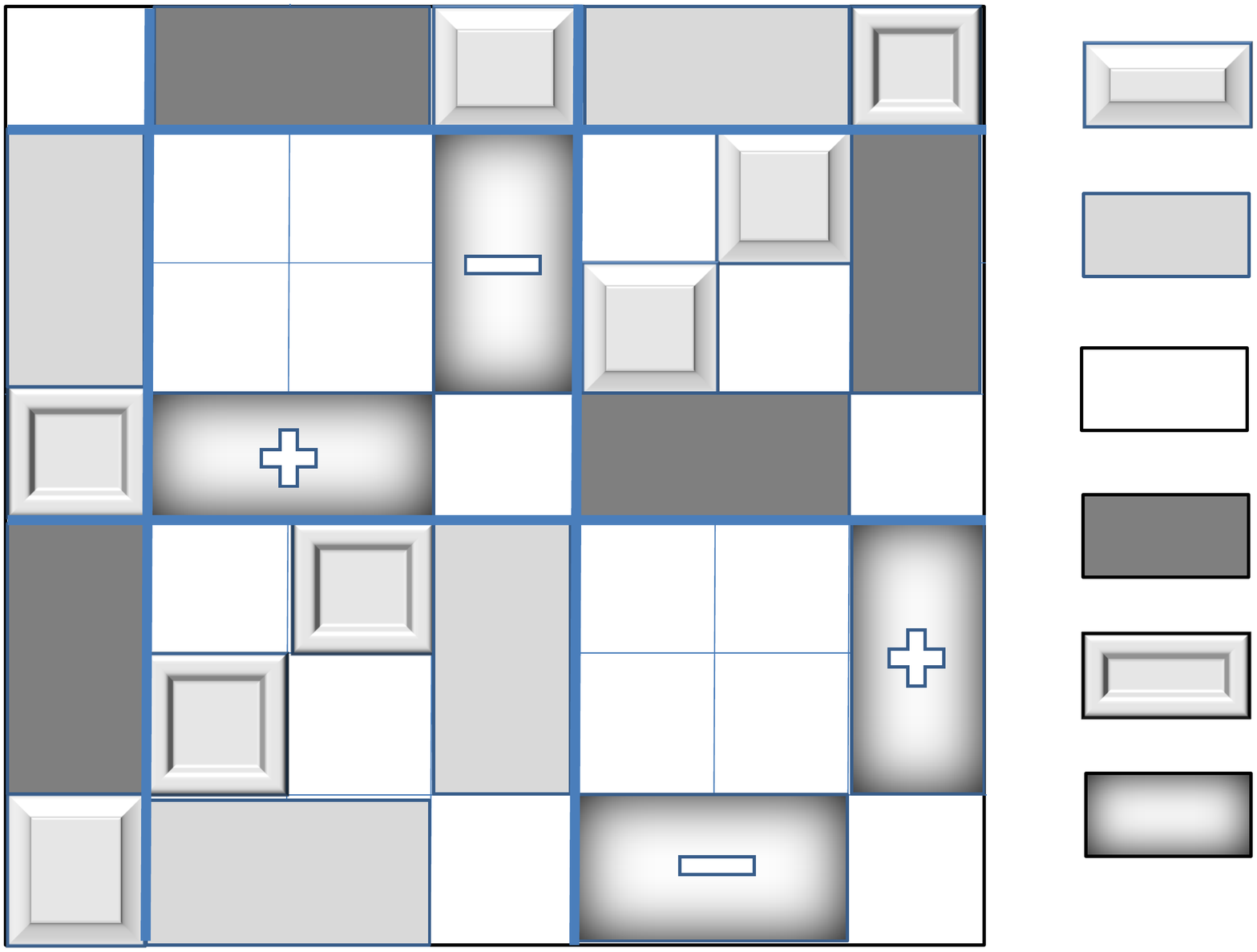}
\begin{picture}(30,0)
\put(-163,105){\small $0$}
\put(-95,90){\small $0$}
\put(-79,75){\small $0$}
\put(-60,57){\small $0$}
\put(-147,42){\small $0$}
\put(-130,25){\small $0$}
\put(-112,8){\small $0$}
\put(-10,102){\text{--}\ $\g_{-2}$}
\put(-10,85){\text{--}\ $\g_{-1}$}
\put(-10,68){\text{--}\ $\g_0$}
\put(-10,51){\text{--}\ $\g_1$}
\put(-10,31){\text{--}\ $\g_2$}
\put(-10,14){\text{--}\ $\g_{\pm 3}$}
\end{picture}
\caption{Case of $G_2$: depth $3$}\label{G2c}
\end{figure}


}

\end{document}